\documentclass[letterpaper, 10pt, conference]{ieeeconf}

\IEEEoverridecommandlockouts                        
 \usepackage{graphicx}

\usepackage{wrapfig} 
\usepackage{times} 
\usepackage{amsmath}
\usepackage{amsmath}
\usepackage{subfig}

\usepackage{amssymb}  
\usepackage{dsfont} 
\usepackage[dvipsnames,usenames]{color}
\usepackage[colorlinks,%
linkcolor=BrickRed,%
filecolor =BrickRed,%
citecolor=RoyalPurple,%
]{hyperref}

\usepackage{amsmath} 
\usepackage{amssymb}  
\usepackage{color}
\usepackage{cite}

\newtheorem{lemma}{Lemma}

\newtheorem{assumption}{Assumption}
\newtheorem{proposition}{Proposition}
\newtheorem{corollary}{Corollary}

\newcommand{\an}[1]{{\color{black}#1}}
\newcommand{\fy}[1]{{\color{black}#1}}

 \newcommand{\remove}[1]{}
\newcommand{\EXP}[1]{\mathsf{E}\!\left[#1\right] }

\def\sF{\mathcal{F}}

\def\Real{\mathbb{R}}
\def\g{\gamma}

\def\e{\epsilon}
\def\a{\alpha}
 \def\n{\nu}

\usepackage{verbatim}
\usepackage{ulem}
\author{Farzad~Yousefian,   
        Angelia~Nedi\'c, and   
		Uday V.~Shanbhag \thanks{Yousefian and Nedi\'{c} are with the Department of Industrial and Enterprise
Systems Engineering, University of Illinois, Urbana, IL 61801, USA,
{\tt\small \{yousefi1,angelia\}@illinois.edu}. Shanbhag is with the Department of Industrial and Manufacturing Engineering, Pennsylvania State University, University Park, PA 16802, USA, {\tt\small udaybag@psu.edu}. Nedi\'{c} and
Shanbhag gratefully acknowledge the support of the NSF through the award
NSF CMMI 0948905 ARRA.}}

\title{\LARGE \bf A {distributed} adaptive steplength stochastic approximation
	method \\ for  monotone stochastic Nash Games}
\begin{document}
\maketitle
\thispagestyle{empty}
\pagestyle{plain}
\vspace{-0.5in}
\begin{abstract}
\an{We} consider a distributed stochastic approximation \fy{(SA)} scheme
for computing an equilibrium of a stochastic Nash game. Standard
\fy{SA} schemes employ diminishing steplength
sequences that are square summable but not summable. 
\an{Such requirements provide a little or no guidance for how to} leverage
Lipschitzian and monotonicity properties of the problem and naive
choices (such as $\gamma_k = 1/k$) generally do not preform uniformly
well on a breadth of problems. {While a centralized adaptive stepsize
SA scheme is proposed in \cite{Farzad1} for the optimization framework,
   such a scheme provides no freedom for the agents in choosing their
	   own stepsizes.  Thus, \an{a} direct application of centralized
	   stepsize schemes is impractical in solving Nash games.}
	   Furthermore, extensions to game-theoretic regimes where players
	   may independently choose steplength sequences are limited to
	   recent work by Koshal et al.~\cite{koshal10single}.  Motivated by
	   these shortcomings, we present a distributed algorithm in which
	   each player updates his steplength based on the previous
	   steplength and \an{some} problem parameters. The steplength rules
	   are derived from minimizing an upper bound {of} \an{the errors
		   associated with players' decisions.  It is shown that these
			   rules generate sequences that  converge} almost
			   surely to an equilibrium of the stochastic Nash game.
			   Importantly, variants of this rule are suggested where
			   players independently select steplength sequences while
			   abiding by an overall coordination requirement.
			   Preliminary numerical results are seen to be promising.
			   \end{abstract} \maketitle

\section{Introduction}\label{sec:introduction}

We consider a class of stochastic Nash games in which every player
solves a stochastic convex program parametrized by adversarial
strategies. Consider an $N$-person stochastic Nash game in which the
$i$th player solves the parametrized convex problem
\begin{equation}\label{eqn:problem}
\min_{x \in X_i}  \EXP{f_i(x_i,x_{-i},\xi_i)},
\end{equation}
where $x_{-i}$ denotes the collection $\{x_j,j\neq i\}$ of decisions of all players other than player $i$. 
For each $i$, the vector $\xi_i:\Omega_i \rightarrow \Real^{n_i}$ is a random vector with a
probability distribution on some set,  
while the function $\EXP{f_i(x_i,x_{-i},\xi_i)}$ is \fy{strongly} convex in $x_i$ for
all $x_{-i} \in \prod_{j \neq i} X_j$. For every $i$, the set $X_i \subseteq\Real^{n_i}$ is closed and convex. We focus on the resulting
stochastic variational inequality (VI) and consider the development of
distributed stochastic approximation schemes that rely on adaptive
steplength sequences. Stochastic
approximation techniques have a long tradition. First proposed by Robbins
and Monro~\cite{robbins51sa} for differentiable functions and
Ermoliev~\cite{Ermoliev76,Ermoliev83,Ermoliev88}, significant effort has
been applied towards theoretical and algorithmic examination of such
schemes (cf.~\cite{Borkar08,Kush03}). Yet, there has been markedly
little on the application of such techniques to solution of stochastic
variational inequalities.  Exceptions include the work by Jiang and
Xu~\cite{Houyuan08}, and more recently by Koshal~et
al.~\cite{koshal10single}. The latter, in particular, develops a single
timescale stochastic approximation scheme for precisely the class of
problems being studied here viz. monotone stochastic Nash games.

Standard stochastic approximation schemes provide little guidance
regarding the choice of a steplength sequence, apart from requiring that
the sequence, denoted by $\{\gamma_k\}$, satisfies 
$ \sum_{k=1}^{\infty} \gamma_k = \infty \mbox{ and }
\sum_{k=1}^{\infty} \gamma_k^2 < \infty.$ This paper is motivated by
the need to develop {\em adaptive} steplength sequences that can be {\em
independently} chosen by players under a limited coordination, while
guaranteeing the overall convergence of the scheme. Adaptive stepsizes
have been effectively used in gradient and subgradient algorithms.
Vrahatis et al. \cite{Vrahatis00} presented a class of gradient  
algorithms with adaptive stepsizes for unconstrained minimization. 
Spall~\cite{Spall98} developed a
general adaptive SA algorithm based on using a simultaneous perturbation
approach for estimating the Hessian matrix. Cicek et al.
\cite{Zeevi11} considered the Kiefer-Wolfowitz (KW) SA algorithm and
derived general upper bounds on its mean-squared error, together with an adaptive
version of the KW algorithm. {Ram et al. \cite{Ram09} considered distributed stochastic subgradient algorithms for convex optimization problems and studied the effects of stochastic errors
on the convergence of the proposed algorithm. \fy{Lizarraga et al. \cite{Liz08} considered a family of two person Mutil-Plant game and developed Stackelberg-Nash equilibrium conditions based on the Robust Maximum Principle.} More recently, Yousefian et al.
	\cite{Farzad1,Farzad2} developed centralized adaptive stepsize SA schemes
		for solving stochastic optimization problems and variational inequalities. The main
			contribution of the current paper lies in developing a class of \textit{distributed} adaptive \textit{stepsize rules} for SA scheme in which each agent chooses its own stepsizes without any specific information about other agents stepsize policy. This degree of freedom in choosing the stepsizes has not been addressed in the centralized schemes.}

Before proceeding, we briefly motivate the question of distributed
computation of Nash equilibria from two different standpoints: (i) First,
the Nash game can be viewed as a competitive analog of a stochastic
\fy{multi-user} convex optimization problem of the form
$ \min_{\fy{x \in X}} \sum_{i=1}^N \EXP{ f_i(x_i,x_{-i},\xi_i)} .$
Furthermore, under the assumption that equilibria of the associated
stochastic Nash game are efficient, our scheme provides a
distributed framework for computing solutions to this problem. In such a
setting, we may prescribe that players employ stochastic approximation
schemes since the Nash game represents an engineered construct employed
for computing solutions; (ii) A second perspective is one drawn from a
bounded rationality approach towards distributed computation of Nash
equilibria. A fully rational avenue for computing equilibria suggests
that each player employs a best response mapping in updating strategies,
based on what the competing players are doing. Yet, when faced by
computational or time constraints, players may instead take a gradient
step. We work in precisely this regime but allow for flexibility in
terms of the steplengths chosen by the players. 

In this paper, we consider the solution of a stochastic Nash game
whose equilibria are completely captured by a stochastic variational
inequality with a strongly monotone mapping. Motivated by the need for
efficient distributed simulation methods for computing solutions to such
problems, we present a distributed scheme in which each player employs
an adaptive rule for prescribing steplengths.  Importantly, these rules
can be  implemented with relatively little coordination by any given
player and collectively lead to iterates that are shown to converge to
the unique equilibrium in an almost-sure sense.  

This paper is organized as follows. In Section \ref{sec:formulation}, we
introduce the formulation of a stochastic Nash games in which every
player solves a stochastic convex problem. In Section
\ref{sec:convergence}, we show the almost-sure convergence of the SA
algorithm under specified assumptions. In Section \ref{sec:recursive
scheme}, motivated by minimizing a suitably defined error bound, we
develop an adaptive steplength stochastic approximation framework in
which every player {\em adaptively} updates his steplength.  It is shown
that the choice of adaptive steplength rules can be obtained
independently by each player under a limited coordination.  Finally, in
Section \ref{sec:numerical}, we provide some numerical results from a
stochastic flow management game drawn from a communication network
setting.

\textbf{Notation:} Throughout this paper, a vector $x$ is assumed to be
a column vector. We write $x^T$ to denote the transpose of a vector $x$.
$\|x\|$ denotes the Euclidean vector norm, i.e., $\|x\|=\sqrt{x^Tx}$. We
use $\Pi_X(x)$ to denote the Euclidean projection of a vector $x$ on a
set $X$, i.e., $\|x-\Pi_X(x)\|=\min_{y \in X}\|x-y\|$. Vector $g$ is a
\textit{subgradient} of a convex function $f$ with domain dom$f$ at
$\bar x \in \hbox{dom}f$ when $f(\bar x) +g^T(x-\bar x) \leq f(x)$ for
all $x \in \hbox{dom}f.$ The set of all subgradients of $f$ at $\bar x$ is 
denoted by $\partial f(\bar x)$. We write \textit{a.s.} as the
abbreviation for ``almost surely'', and use $\EXP{z}$ to denote the expectation of a random variable~$z$. 

\section{Problem formulation}\label{sec:formulation}
In this section, we present (sufficient) conditions associated with equilibrium  points of
the stochastic Nash game defined by \eqref{eqn:problem}. 
The equilibrium conditions of this game 
can be characterized by
a stochastic variational inequality problem denoted by VI$(X,F)$, where 
\begin{align}\label{eqn:VI_elements}
F(x)\triangleq \left( \begin{array}{ccc}
\nabla _{x_1} \EXP{f_1(x, \xi_1)}  \\
\vdots \\  \nabla _{x_N} \EXP{f_N(x, \xi_N)} \end{array} \right), \quad X= \prod_{i=1}^N X_i,
\end{align}
with $x \triangleq (x_1,\ldots,x_N)^T$ and $x_i \in X_i{\subseteq \mathbb{R}^{n_i}}$ for
$i=1,\ldots,N$. Given a set $X \subseteq \Real^n$ and a single-valued
mapping $F:X \to \Real^n$, then a vector $x^* \in X$ solves a  variational inequality VI$(X,F)$, if
 \begin{align}\label{eqn:VI}
(x-x^*)^TF(x^*) \geq 0 \hbox{ for all } x \in X.
\end{align} 
\fy{Let $n=\sum_{i=1}^N n_i$, and note that when the sets $X_i$ are convex and closed for all
$i$, the set $X \in \Real^{n}$ is closed and convex.}

In the context of solving the stochastic variational inequality
VI$(X,F)$ in (\ref{eqn:VI_elements})-(\ref{eqn:VI}), suppose each player employs a stochastic
approximation scheme for  given by
\begin{align}
\begin{aligned}
x_{{k+1},i} & =\Pi_{X_i}\left(x_{k,i}-\g_{k,i} ( F_i(x_k)+w_{k,i})\right),\cr 
w_{k,i} &\triangleq \hat F_i(x_k,\xi_k)-F_i(x_k),
\end{aligned}\label{eqn:algorithm_different}
\end{align}
for all $k\ge 0$ and $i=1,\dots,N$, where
$\g_{k,i} >0$ is the stepsize of the $i$th player at iteration $k$,  	
$x_k =(x_{k,1}\  x_{k,2}\ \ldots \ x_{k,N})^T$,
$\xi_k =(\xi_{k,1}\  \xi_{k,2}\ \ldots \ \xi_{k,N})^T$, \fy{$F_i=\EXP{\nabla _{x_i} f_i(x, \xi_i)}$,} and 
\begin{align*}	
\hat F(x,\xi)\triangleq \left( \begin{array}{ccc}
\nabla _{x_1} 	f_1(x, \xi_1) \\
\vdots \\  \nabla _{x_N} {f_N(x		, \xi_N)} \end{array} \right), \quad	
 \xi \triangleq \left( \begin{array}{ccc}
\xi_{1}  \\
\vdots \\  \xi_{N} \end{array} \right).
\end{align*}
\fy{Note that in terms of the definition of $w_{k,i}$, $F_i$, and $\hat F_i$,  $\EXP{w_{k,i}\mid \sF_k}=0.$}
In addition, $x_0 \in X$ is a random
initial vector independent of the random variable $\xi$ and such
that $\EXP{\|x_0\|^2}<\infty$. Note that each player uses its individual
stepsize to update its decision. 

\section{A Distributed SA scheme}\label{sec:convergence}
In this section, we present conditions under which algorithm
\eqref{eqn:algorithm_different} converges almost surely to the solution of game
\eqref{eqn:problem} under suitable assumptions on the
	mapping. 
Also, we develop a distributed variant of a standard
stochastic approximation scheme and provide conditions on the
steplength sequences that lead to almost-sure convergence of the
iterates to the unique solution. Our assumptions include
requirements on the set $X$ and the mapping $F$. 

\begin{assumption}\label{assum:different_main}
Assume the following:
\begin{enumerate} 
\item[(a)] The sets $X_i \subseteq \Real^{n_i}$ are closed and convex.
\item[(b)] \fy{$F(x)$ is {strongly monotone with constant $\eta >0$ and Lipschitz continuous with constant $L$ over the set $X$.}}
\end{enumerate}
\end{assumption}
{\textbf{Remark:} The strong monotonicity is assumed to hold throughout the paper. Although the convergence results may still hold with a weaker assumption, such as strict monotonicity, but the stepsize policy in this paper leverages the strong monotonicity parameter which prescribes a more parametrized stepsize rule. This is the main reason that we assumed the stronger version of monotonicity. In Section \ref{sec:numerical}, we present an example where such an assumption is satisfied.} 

Another set of assumptions is for the stepsizes employed by each
player in algorithm \eqref{eqn:algorithm_different}.
\begin{assumption}\label{assum:step_error_sub}
Assume that:
\begin{enumerate} 
 \item[(a)] The {stepsize sequences are} such that $\gamma_{k,i}>0$ for all $k$ and $i$, with
$\sum_{k=0}^\infty \gamma_{k,i} = \infty$ and $\sum_{k=0}^\infty
\gamma_{k,i}^2 < \infty$.
\item [(b)] There exists a scalar $\beta$ such that $0\le \beta<\frac{\eta}{L}$   and
 $\frac{\Gamma_k-\delta_k}{\delta_k}\le \beta$ for all $k\ge 0$, where $\delta_k$ and $\Gamma_k$ are (fixed) positive sequences satisfying $\delta_k \leq \min_{i=1,\ldots,N}{\g_{k,i}}$ and $\Gamma_k \geq \max_{i=1,\ldots,N}{\g_{k,i}}$ for all $k\ge0$.
\end{enumerate}
\end{assumption}
We let $\sF_k$ denote the history of the method up to time $k$, i.e., 
$\sF_k=\{x_0,\xi_0,\xi_1,\ldots,\xi_{k-1}\}$ for $k\ge 1$ and $\sF_0=\{x_0\}$. {Consider the following assumption on the stochastic errors, $w_k$, of the algorithm.}
\begin{assumption}\label{assum:w_k_bound} 
The errors $w_k$ are such that 
for some {constant} $\nu >0$,
\[\EXP{\|w_k\|^2 \mid \sF_k} \le \nu^2 \qquad \hbox{\textit{a.s.} for all $k\ge0$}.\]  
\end{assumption}
We \fy{use the Robbins-Siegmund lemma} in establishing the convergence of
method~(\ref{eqn:algorithm_different}), which can be found
in~\cite{Polyak87} (cf.~Lemma 10, page 49).  
\begin{lemma}\label{lemma:probabilistic_bound_polyak}
Let $\{v_k\}$ be a sequence of nonnegative random variables, where 
$\EXP{v_0} < \infty$, and let $\{\a_k\}$ and $\{\mu_k\}$
be deterministic scalar sequences such that:
\begin{align*}
& \EXP{v_{k+1}|v_0,\ldots, v_k} \leq (1-\alpha_k)v_k+\mu_k
\qquad a.s. \ \hbox{for all }k\ge0, \cr
& 0 \leq \alpha_k \leq 1, \quad\ \mu_k \geq 0, \cr
& \quad\ \sum_{k=0}^\infty \alpha_k =\infty, 
\quad\ \sum_{k=0}^\infty \mu_k < \infty, 
\quad\ \lim_{k\to\infty}\,\frac{\mu_k}{\alpha_k} = 0. 
\end{align*}
Then, $v_k \rightarrow 0$ almost surely. 
\end{lemma} 

The following lemma provides an error bound for algorithm \eqref{eqn:algorithm_different} {under Assumption \ref{assum:different_main}.}
\begin{lemma}\label{lemma:error_bound_1}
Consider algorithm~\eqref{eqn:algorithm_different}. 
Let Assumption \ref{assum:different_main} hold. Then, the following relation holds \fy{a.s.} for all $k \ge 0$:
\begin{align}\label{ineq}
&\EXP{\|x_{k+1}-x^*\|^2\mid\sF_k} \leq 	\Gamma_k^2\EXP{\|w_k\|^2\mid\sF_k}\cr
&  +(1-2(\eta+L)\delta_k+2L\Gamma_k+L^2\Gamma_k^2)\|x_	k-x^*\|^2.
\end{align} 
\end{lemma}
\begin{proof}
By  Assumption \ref{assum:different_main}a, 
\an{the set $X$ is closed and convex}. Since $F$ is strongly monotone, the existence and uniqueness of the solution to VI$(X,F)$ is guaranteed by Theorem 2.3.3 of~\cite{facchinei02finite}.
Let $x^*$ denote the solution of VI$(X,F)$. From properties of projection operator, we know that a vector $x^*$ solves VI$(X,F)$ problem if and only if $x^*$ satisfies
\[x^*=\Pi_X(x^*-\g F(x^*))\qquad\hbox{for any }\g>0.\]
\fy{From} algorithm \eqref{eqn:algorithm_different} and the non-expansiveness property of the
projection operator, we have for all $k\ge0$ and~$i$,
\begin{align*}
 &\|x_{k+1,i}-x_i^*\|^2  =\|\Pi_{X_i}(x_{k,i}-\g_{k,i}(F_i(x_k)+w_{k,i})) \\
 & -\Pi_{X_i}(x_i^*-\g_{k,i} F_i(x^*))\|^2 \\
 & \le \|x_{k,i}-x_i^*-\g_{k,i}(F_i(x_k)+w_{k,i}-F_i(x^*))\|^2.
\end{align*}	
Taking the expectation conditioned on the past, and using
$\EXP{w_{k,i}\mid \sF_k}=0$, we have
\begin{align*}
& \EXP{\|x_{k+1,i}-x_i^*\|^2\mid\sF_k} \le \|x_{k,i}-x_i^*\|^2 \cr
&+\gamma_{k,i}^2\|F_i(x_k)-F_i(x^*)\|^2  +\g_{k,i}^2\EXP{\|w_{k,i}\|^2\mid \sF_k}\cr 
&-2\gamma_{k,i} (x_{k,i}-x_i^*)^T(F_i(x_k)-F_i(x^*)).
\end{align*}
Now, by summing the preceding relations over $i$,  we have
\begin{align}\label{ineq:main1}
& \EXP{\|x_{k+1}-x^*\|^2\mid\sF_k} \le \|x_{k}-x^*\|^2 \cr
&+\underbrace{\sum_{i=1}^N\gamma_{k,i}^2\|F_i(x_k)-F_i(x^*)\|^2}_{\bf Term\,1} +\sum_{i=1}^N\g_{k,i}^2
\EXP{\|w_{k,i}\|^2\mid \sF_k}\cr 
&\underbrace{-2\sum_{i=1}^N\gamma_{k,i} (x_{k,i}-x_i^*)^T(F_i(x_k)-F_i(x^*))}_{\bf Term\,2}.
\end{align}
Next, we estimate Term $1$ and Term $2$ in (\ref{ineq:main1}).
By using the definition of $\Gamma_k$ and by leveraging the Lipschitzian
	property of mapping $F$, we obtain
\begin{align}\label{ineq:main2}
\hbox{Term}\,1 \le  \Gamma_k^2\|F(x_k)-F(x^*)\|^2\le\Gamma_k^2L^2\|x_k-x^*\|^2. 
\end{align}
Adding and subtracting
$-2\sum_{i=1}^N\delta_k(x_{k,i}-x_i^*)^T(F_i(x_k)-F_i(x^*))$ {from} Term $2$, we further obtain
\begin{equation*}
\begin{split}
\hbox{Term}\,2  \le 
 & -2\delta_k(x_{k}-x^*)^T(F(x_k)-F(x^*))\cr
& -2\sum_{i=1}^N(\g_{k,i}-\delta_k)(x_{k,i}-x_i^*)^T{(F_i(x_{k})-F_i(x^*))}.
\end{split}
\end{equation*} 
By the Cauchy-Schwartz inequality, \fy{we obtain}
\begin{equation*}
\begin{split}
\hbox{Term}\,2  \le & -2\delta_k(x_{k}-x^*)^T(F(x_k)-F(x^*))\cr
& +2(\g_{k,i}-\delta_k)\sum_{i=1}^N\|x_{k,i}-x_i^*\|\|{F_i(x_{k})-F_i(x^*)}\|\cr
 \le & -2\delta_k(x_{k}-x^*)^T(F(x_k)-F(x^*))\cr
& +2(\Gamma_k-\delta_k)\|x_{k}-x^*\|\|F(x_{k})-F(x^*)\|,
\end{split}
\end{equation*}
where in the last relation, we use {H\"{o}lder's} inequality.
Invoking the strong monotonicity of the mapping for bounding the first term and by
{utilizing the Lipschitzian property of} the second term of the preceding relation, we have
\begin{equation*}
\begin{split}
\hbox{Term}\,2  \le -2\eta\delta_k\|x_{k}-x^*\|^2 +2(\Gamma_k-\delta_k)L\|x_{k}-x^*\|^2.
\end{split}
\end{equation*}
The desired inequality is obtained by combining relations
(\ref{ineq:main1}) and (\ref{ineq:main2}) with the preceding inequality .
\end{proof}

We next prove that algorithm
\eqref{eqn:algorithm_different} generates a sequence of iterates that
converges \fy{a.s.} to the unique solution of VI$(X,F)$,  as seen in the following proposition.
Our proof of this result makes use of Lemma~\ref{lemma:error_bound_1}.

\begin{proposition}[Almost-sure convergence]
\label{prop:rel_bound}
Consider the algorithm \ref{eqn:algorithm_different}. Let Assumption \ref{assum:different_main}, \ref{assum:step_error_sub} and \ref{assum:w_k_bound} hold. Then,
\begin{enumerate}
\item [(a)] The following relation holds \fy{a.s.} for all $k \ge 0$:
\begin{align*}
&\EXP{\|x_{k+1}-x^*\|^2} \leq (1+\beta)^2\delta_k^2\nu^2\cr
&  +(1-2(\eta-\beta L)\delta_k+(1+\beta)^2L^2\delta_k^2)\EXP{\|x_	k-x^*\|^2}.
\end{align*} 
\item [(b)] The sequence $\{x_k\}$
generated by algorithm (\ref{eqn:algorithm_different}), converges \fy{a.s.}
to the unique solution of VI$(X,F)$.
\end{enumerate}
\end{proposition}
\begin{proof}
(a) \ Assumption \ref{assum:step_error_sub}b implies that $\Gamma_k \leq
(1+\beta)\delta_k$. Combining this with inequality \eqref{ineq}, we obtain
\begin{align*}
&\EXP{\|x_{k+1}-x^*\|^2\mid\sF_k} \cr
&\leq  (1-2(\eta-\beta L)\delta_k+(1+\beta)^2L^2\delta_k^2)\|x_	k-x^*\|^2\cr 
&+(1+\beta)^2\delta_k^2\EXP{\|w_k\|^2\mid\sF_k}, \qquad \hbox{for all } k \ge 0.
\end{align*}  
Taking expectations in the preceding inequality and using Assumption
\ref{assum:w_k_bound}, we obtain the desired relation.

\noindent (b) \ We {show that the conditions of} Lemma
\ref{lemma:probabilistic_bound_polyak} {are satisfied in order} to
claim almost sure convergence of $x_k$ to $x^*$. Let us define $v_k
\triangleq \|x_{k+1}-x^*\|^2$, $\alpha_k \triangleq 2(\eta-\beta
		L)\delta_k-L^2\delta_k^2(1+\beta)^2$, and $\mu_k \triangleq
(1+\beta)^2\delta_k^2\EXP{\|w_k\|^2\mid\sF_k}$. Since $\g_{k,i}$
tends to zero for any $i=1,\ldots,N$, we {may} conclude that $\delta_k$ goes to zero as $k$ grows. 
\an{Recall that} $\alpha_k$ is given by 
\[\alpha_k =2(\eta-\beta L)\delta_k\left(1-\frac{(1+\beta)^2L^2\delta_k}{2(\eta-\beta
						L)}\right).\]
\an{Due to $\delta_k\to0$, for all} $k$ large enough, say $k>k_1$, we have
\[1-\frac{(1+\beta)^2L^2\delta_k}{2(\eta-\beta L)}>0.\]
Since $\beta < \frac{\eta}{L}$ (Assumption \ref{assum:step_error_sub}b), 
it follows $\eta-\beta L>0$. Thus,  we have $\alpha_k \geq 0$. Also, for $k$ large enough, say $k>k_2$, we have $\alpha_k \leq 1$. Therefore, when $k>\max\{k_1,k_2\}$ we have $0 \leq \alpha_k \leq 1$. Obviously, $v_k, \mu_k \geq 0$. From Assumption~\ref{assum:step_error_sub}a and 
Assumption~\ref{assum:w_k_bound} it follows  $\sum_{k}\mu_k < \infty$. We also have 
\begin{align*}
{\lim_{k \rightarrow	\infty}} \frac{\mu_k}{\alpha_k}&={\lim_{k \rightarrow	\infty}} \frac{(1+\beta)^2\delta_k^2\EXP{\|w_k\|^2\mid\sF_k}}{2(\eta-\beta L)\delta_k\left(1-\frac{(1+\beta)^2L^2\delta_k}{2(\eta-\beta L)}\right)}\cr &= {\lim_{k \rightarrow	\infty} }\frac{(1+\beta)^2\delta_k\EXP{\|w_k\|^2\mid\sF_k}}{2(\eta-\beta L)}.\end{align*} 
Since the term $\EXP{\|w_k\|^2\mid\sF_k}$ is bounded by $\nu^2$ 
(Assumption~\ref{assum:w_k_bound}) and
\an{$\delta_k\to0$, we see} that $\lim_{k \rightarrow
	\infty}\frac{\mu_k}{\alpha_k}=0$. Hence, the conditions of Lemma \ref{lemma:probabilistic_bound_polyak} are satisfied, which implies that $x_k$ converges to 
the unique solution, $x^*$, almost surely. 
\end{proof} 

Consider now a special form of algorithm~\eqref{eqn:algorithm_different} corresponding to
the case when all players employ the same stepsize, i.e., $\gamma_{k,i}=\gamma_k$ for all $k$.
Then, the algorithm~\eqref{eqn:algorithm_different} reduces to the following:
\begin{align}
\begin{aligned}
 x_{k+1} & =\Pi_{X}\left(x_k-\g_k( F(x_k)+w_k)\right),\cr 
 w_k & \triangleq \hat F(x_k,\xi_k)-F(x_k),
\end{aligned}\label{eqn:algorithm_identical}
\end{align}
for all $k \ge 0$. 
Observe that when $\gamma_{k,i}=\gamma_k$ for all $k$, Assumption~\ref{assum:step_error_sub}a
is satisfied when 
$\sum_{k=0}^\infty \gamma_k = \infty$ and $\sum_{k=0}^\infty\gamma_k^2 < \infty$.
Assumption~\ref{assum:step_error_sub}b is automatically satisfied 
with $\Gamma_k=\delta_k=\gamma_k$ and $\beta=0$. 
Hence, as a direct consequence of Proposition \ref{prop:rel_bound}, we have the following corollary.

\begin{corollary}[Identical stepsizes]
\label{corollary:rel_bound}
Consider algorithm (\ref{eqn:algorithm_identical}). Let Assumption \ref{assum:different_main} and
\ref{assum:w_k_bound} hold. Also, let
$\sum_{k=0}^\infty \gamma_k = \infty$ and $\sum_{k=0}^\infty\gamma_k^2 < \infty$.
Then,
\begin{enumerate}
\item [(a)] The following relation holds almost surely:
{\small 
\begin{align*}
\EXP{\|x_{k+1}-x^*\|^2} & \leq 
 (1-2	\eta\g_k+L^2\g_k^2 ) 
\EXP{\|x_{k}-x^*\|^2} \cr 
&+\g_k^2 \nu^2.
\end{align*}}
\item [(b)] The sequence $\{x_k\}$
generated by algorithm (\ref{eqn:algorithm_identical}), converges \fy{a.s.} to the unique solution of VI$(X,F)$.
\end{enumerate}
\end{corollary}

\section{A distributed adaptive steplength SA scheme}\label{sec:recursive scheme}
Stochastic approximation algorithms 
require stepsize sequences to be square summable
but not summable. These algorithms provide little advice regarding the choice of
such sequences. One of the most common choices has been the harmonic
steplength rule which takes the form of $\gamma_k=\frac{\theta}{k}$ where $\theta>0$ is a
constant. Although, this choice guarantees almost-sure convergence, it does
not leverage problem parameters. Numerically, it has been observed that
such choices can perform quite poorly in practice. Motivated by this
shortcoming, we present a
distributed adaptive steplength scheme for algorithm
(\ref{eqn:algorithm_different}) which guarantees almost-sure
convergence of $x_k$ to the unique solution of VI$(X,F)$. It is derived
from the minimizer of a suitably defined error bound and leads to a
recursive relation; more specifically, at each step, the
new stepsize is calculated using the stepsize from the preceding iteration and
problem parameters. To begin our analysis, we consider the
result of  Proposition \ref{prop:rel_bound}a for all $k \geq 0$:
\begin{align}\label{equ:bound_recursive_01}
&\EXP{\|x_{k+1}-x^*\|^2} \leq (1+\beta)^2\delta_k^2\nu^2\cr
&+(1-2(\eta-\beta L)\delta_k+(1+\beta)^2L^2\delta_k^2)\EXP{\|x_	k-x^*\|^2}.\quad
\end{align} 
When the stepsizes are further restricted so that \[0 < \delta_k
\le\frac{\eta-\beta L}{(1+\beta)^2L^2},\] we have
\begin{align*}
1-2(\eta-\beta L)\delta_k  +L^2(1+\beta)^2\delta_k^2  \leq 1-(\eta -\beta L) \delta_k.
\end{align*}
Thus, for  $0 < \delta_k \le\frac{\eta-\beta L}{(1+\beta)^2L^2}$, from inequality (\ref{equ:bound_recursive_01}) we obtain
{\small \begin{align}\label{rate_nu_est_different}
\EXP{\|x_{k+1}-x^*\|^2} 
& \le (1-(\eta-\beta L) \delta_k)\EXP{\|x_k-x^*\|^2}\cr
& + (1+\beta)^2\delta_k^2\nu^2\qquad\hbox{for all }k \geq 0.\quad
\end{align}} \hskip -0.3pc
Let us view the quantity
$\EXP{\|x_{k+1}-x^*\|^2}$ as an error $e_{k+1}$ of the method arising
from the use of the stepsize values $\delta_0,\delta_1,\ldots, \delta_k	$. Relation~\eqref{rate_nu_est_different} gives us an estimate of the error of algorithm (\ref{eqn:algorithm_different}). We use this estimate to develop an adaptive stepsize
procedure.  Consider the worst case which is the case when ~\eqref{rate_nu_est_different} holds with equality.  In this worst case,
	 the error satisfies the following recursive relation:
\begin{align*}
& e_{k+1} = (1-(\eta-\beta L) \delta_k) e_k+(1+\beta^2)\delta_k^2 \nu^2.
\end{align*}
Let us assume that we want to run the algorithm
(\ref{eqn:algorithm_different}) for a fixed number of iterations, say
{$K$}. The preceding relation shows that $e_{K}$ depends on the stepsize
values up to the $K$th iteration. This motivates us to see the stepsize
parameters as decision variables that can minimize \fy{a suitably defined} error bound of the
algorithm. Thus, the variables are {$\delta_0, \delta_1, \ldots,
	\delta_{K-1}$} and the objective function is the error function
	$e_K(\delta_0, \delta_1, \ldots, \delta_{K-1})$. We proceed to derive a stepsize rule by minimizing the error $e_{K+1}$;
	Importantly, $\delta_{K+1}$ can be shown to be a function of only the most
	recent stepsize $\delta_{K}$. We define the real-valued
	error function $e_k(\delta_0, \delta_1, \ldots, \delta_{k-1})$ by
	the upper bound in \eqref{rate_nu_est_different}: 
\begin{align}\label{def:e_k_different}
 e_{k+1}(\delta_0,\ldots,\delta_k) \triangleq &(1-(\eta-\beta L) \delta_k) e_k(\delta_0,\ldots,\delta_{k-1})\cr 
& +(1+\beta^2)\delta_k^2 \nu^2\qquad\hbox{for all }k \geq 0,
\end{align}
where $e_0$ is a positive scalar, $\eta$ is the strong monotonicity parameter 
and $\nu^2$ is the upper bound for the second moments of the error 
norms $\|w_k\|$.

Now, let us consider the stepsize sequence $\{\delta^*_k\}$ given by 
\begin{align}
& \delta_0^*=\frac{\eta-\beta L}{2(1+\beta)^2\nu^2}\,e_0\label{eqn:gmin0}
\\
& \delta_k^*=\delta_{k-1}^*\left(1-\frac{\eta-\beta L}{2}\delta_{k-1}^*\right) 
\quad \hbox{for all }k\ge 1.\label{eqn:gmink}
\end{align}
In what follows, we often abbreviate $e_k(\delta_0,\ldots,\delta_{k-1})$ by $e_k$
whenever this is unambiguous. The next proposition shows that the
lower bound sequence of $\gamma_{k,i}$ given by
(\ref{eqn:gmin0})--(\ref{eqn:gmink}) minimizes the errors $e_k$ over $(0,\frac{\eta -\beta L}{(1+\beta)^2L^2}]^{k}$.
\begin{proposition}\label{prop:rec_results}
Let $e_k(\delta_0,\ldots,\delta_{k-1})$ be defined as in~\eqref{def:e_k_different},
where $e_0>0$ is such that $\,e_0<\frac{2\nu^2}{L^2}$,
and $L$ is the Lipschitz constant of mapping $F$.
Let the sequence $\{\delta^*_k\}$ be given by 
\eqref{eqn:gmin0}--\eqref{eqn:gmink}.
Then, the following hold:
\begin{itemize}
\item [(a)] \fy{$e_k(\delta_0^*,\ldots,\delta_k^*) = \frac{2(1+\beta)^2\nu^2}{\eta-\beta L}\,\delta_k^*$ for all $k \geq 0$.}
\item [(b)] For any $k\ge 1$, the vector $(\delta_0^*, \delta_1^*,\ldots,\delta_{k-1}^*)$ is the minimizer of the function $e_k(\delta_0,\ldots,\delta_{k-1})$ over the set
$$\mathbb{G}_k 
\triangleq \left\{\alpha \in \Real^k : 0< \alpha_j \leq \frac{\eta-\beta L}{(1+\beta)^2L^2}, j =1,\ldots, k\right\},$$
\fy{i.e., for any $k\ge 1$ and $(\delta_0,\ldots,\delta_{k-1})\in \mathbb{G}_k$:}
	\begin{align*}
	& e_{k}(\delta_0,\ldots,\delta_{k-1}) 
         - e_{k}(\delta_0^*,\ldots,\delta_{k-1}^*)\cr  
          &\geq (1+\beta)^2\nu^2(\delta_{k-1}-\delta_{k-1}^*)^2.
	\end{align*}	
\end{itemize}
\end{proposition} 

\begin{proof}

(a) \ To show the result, we use induction on $k$. 
Trivially, it holds for 
$k=0$ from \eqref{eqn:gmin0}. Now, suppose that we have
$e_k(\delta_0^*,\ldots,\delta_{k-1}^*) = \frac{2(1+\beta)^2\nu^2}{\eta-\beta L}\,\delta_k^*$ for some $k$,
and consider the case for $k+1$. From the definition of the error $e_k$ 
in~\eqref{def:e_k_different} \fy{and the inductive hypothesis, we have
\begin{align*}
 e_{k+1}(\delta_0^*,\ldots,\delta_k^*)  
&=(1-(\eta-\beta L)\delta_k^*)\frac{2(1+\beta)^2\nu^2}{\eta-\beta L}\,\delta_k^*\cr
&+(1+\beta)^2(\delta_k^*)^2\nu^2\cr
&=\frac{2(1+\beta)^2\nu^2}{\eta-\beta L}\,\delta_k^*\left(1-\frac{\eta-\beta L}{2}\,\delta_k^*\right) \cr
&=\frac{2(1+\beta)^2\nu^2}{\eta-\beta L}\,\delta_{k+1}^*,\end{align*}}

where the last equality follows by the definition of $\delta_{k+1}^*$ in
\eqref{eqn:gmink}. Hence, the result holds for all $k \ge 0$.

\noindent (b) \ {First we need to show that
	$(\delta^*_{0},\ldots,\delta_{k-1}^*)\in\mathbb{G}_k$. By the choice
		of $e_0$, i.e. $e_0 < \frac{2\nu^2}{L^2}$, we have that $0<\delta_0^*\leq \frac{\eta- \beta L}{(1+\beta)^2 L^2}$.
Using induction, from relations (\ref{eqn:gmin0})--(\ref{eqn:gmink}), it can be shown that $0< \delta_{k}^*<\delta^*_{k-1}$ for all $k\ge1$. 
Thus, $(\delta^*_{0},\ldots,\delta_{k-1}^*)\in\mathbb{G}_k$ for all $k\ge1$. }Using induction on $k$, we now show that vector $(\delta_0^*,\delta_1^*,\ldots,\delta_{k-1}^*)$  
minimizes the error $e_k$ for all $k\ge1$. From the definition of the error $e_1$ and the relation
\[e_1(\delta_0^*)=\frac{2(1+\beta)^2\nu^2}{\eta-\beta L}\,\delta_1^*\] shown in part (a), we have
\begin{align*}
e_1(\delta_0) -e_1(\delta_0^*)
 &=(1-(\eta-\beta L) \delta_0)e_0+(1+\beta)^2\nu^2 \delta_0^2\cr
&-\frac{2(1+\beta)^2\nu^2}{\eta-\beta L}\,\delta_1^*.
\end{align*}
Using $\delta_1^*=\delta_0^*\left(1-\frac{\eta-\beta L}{2}\,\delta_0^*\right)$, we obtain
\begin{align*}
e_1(\delta_0) -e_1(\delta_0^*)
 &=(1-(\eta-\beta L) \g_0)e_0+(1+\beta)^2\nu^2 \delta_0^2\cr
&-\frac{2(1+\beta)^2\nu^2}{\eta-\beta L}\,\delta_0^*+(1+\beta)^2\nu^2(\delta_0^*)^2.
\end{align*}
where the last equality follows from $e_0=\frac{2(1+\beta)^2\nu^2}{\eta-\beta L}\,\delta_0^*.$
Thus, we have
\begin{align*}
 e_1(\delta_0) -e_1(\delta_0^*) 
 & = (1+\beta)^2\nu^2\left(-2\delta_0\delta_0^*+\delta_0^2+(\delta_0^*)^2 \right)\cr
& =(1+\beta)^2\nu^2\left(\delta_0-\delta_0^*\right)^2,
\end{align*}
and the inductive hypothesis holds for $k = 1$. 
Now, suppose that $e_k(\delta_0,\ldots,\delta_{k-1}) 
\geq e_k(\delta_0^*,\ldots,\delta_{k-1}^*)$ holds for some $k$ and
any $(\delta_0,\ldots,\delta_{k-1}) \in\mathbb{G}_k$, and we \fy{need} to show that 
$e_{k+1}(\delta_0,\ldots,\delta_k) \geq e_{k+1}(\delta_0^*,\ldots,\delta_k^*)$ 
holds for all $(\delta_0,\ldots,\delta_k) \in\mathbb{G}_{k+1}$. 
To simplify the notation, 
we use $e_{k+1}^*$ to denote the error $e_{k+1}$ evaluated at 
$(\delta_{0}^*,\delta_1^*,\ldots,\delta_k^*)$, and $e_{k+1}$ when 
evaluating at an arbitrary vector 
$(\delta_0,\delta_1,\ldots,\delta_k)\in\mathbb{G}_{k+1}$.
Using (\ref{def:e_k_different})  and part (a), we have
\begin{align*}
e_{k+1}- e_{k+1}^*
& = (1-(\eta-\beta L) \delta_k)e_k
+(1+\beta)^2\nu^2\delta_k^2\cr
& -\frac{2(1+\beta)^2\nu^2}{\eta-\beta L} \delta_{k+1}^*.
\end{align*}
Under the inductive hypothesis, we have $e_k\ge e_k^*$. It can be shown easily that when $(\delta_0,\delta_1,\ldots,\delta_k)\in \mathbb{G}_k$, we have $0<1-(\eta-\beta L)\delta_k< 1$. 
Using this, the relation $e_k^*=\frac{2(1+\beta)^2\nu^2}{\eta-\beta L}\g_k^*$
of part (a), and the definition of $\delta_{k+1}^*$, we obtain
\begin{align*}
e_{k+1} - e_{k+1}^* 
  & \geq  (1-(\eta-\beta L) \delta_k)\frac{2(1+\beta)^2\nu^2}{\eta-\beta L}\delta_k^*\cr 
 &+(1+\beta)^2\nu^2\delta_k^2 \cr
&-\frac{2(1+\beta)^2\nu^2}{\eta-\beta L} \delta_k^*\left(1-\frac{\eta-\beta L}{2}\delta_k^*\right) 
 \cr 
 &=(1+\beta)^2 \nu^2(\delta_k-\delta_k^*)^2.
\end{align*}
Hence, {$e_{k} - e_{k}^*\ge
(1+\beta)^2\nu^2(\delta_{k-1}-\delta_{k-1}^*)^2$ holds} for all $k\ge1$ and all 
$(\delta_0,\ldots,\delta_{k-1})\in\mathbb{G}_k$. 
\end{proof}

We have just provided an analysis in terms of the lower bound
sequence $\{\delta_{k}\}$. We can conduct a similar analysis for $\{\Gamma_{k}$\}
and obtain the
corresponding adaptive stepsize scheme using the following relation: 
\begin{align*}
&\EXP{\|x_{k+1}-x^*\|^2} \leq \Gamma_k^2\nu^2 \cr &  +(1-\frac{2(\eta+L)}{1+\beta}\Gamma_k+2L\Gamma_k
+L^2\Gamma_k^2)\EXP{\|x_	k-x^*\|^2}.
\end{align*}	
When $0 < \Gamma_k
\le\frac{\eta-\beta L}{(1+\beta)L^2}$, we have
\begin{align}\label{rate_nu_est_different_max}
\EXP{\|x_{k+1}-x^*\|^2} 
& \le (1-\frac{(\eta-\beta L)}{1+\beta} \Gamma_k)\EXP{\|x_k-x^*\|^2}\cr
&+ \Gamma_k^2\nu^2\qquad\hbox{for all }k \geq 0.
\end{align}
Using relation (\ref{rate_nu_est_different_max}) and following similar approach in Proposition \ref{prop:rec_results}, we obtain the sequence $\{\Gamma^*_{k}\}$ given by 
\begin{align}
& \Gamma_{0}^*=\frac{\eta-\beta L}{2(1+\beta)\nu^2}\,e_0\label{eqn:gmax0}
\\
& \Gamma_k^*=\Gamma_{k-1}^*\left(1-\frac{\eta-\beta L}{2(1+\beta)}\Gamma_{k-1}^*\right)\qquad \hbox{for all } k \ge 1. 
\label{eqn:gmaxk}
\end{align}
\fy{Note that the adaptive stepsize sequence given by (\ref{eqn:gmax0})--(\ref{eqn:gmaxk}) converges to zero and moreover, it is not summable but squared summable (cf. \cite{Farzad1}, Proposition $3$).}
In the following lemma, we derive a relation between two recursive sequences, 
which we use later to obtain our main recursive stepsize scheme. 
\begin{lemma}\label{lemma:two-rec}
Suppose that sequences 
$\{\lambda_k\}$ and $\{\g_k\}$ are given with the following recursive equations 
for all  $k\geq 0$,
\begin{align}
\lambda_{k+1}& =\lambda_{k}(1-\lambda_{k}),  \hbox{ and } \g_{k+1} & =\g_{k}(1-c\g_{k}), \notag
\end{align}
where $\lambda_0=c\g_0$, $0<\g_0<\frac{1}{c}$, and $c>0$.
Then for all $k \geq 0$,
\[\lambda_{k}=c\g_k.\]
\end{lemma}
\begin{proof}
We use induction on $k$. For $k=0$, the relation holds since $\lambda_0=c\g_0$. 
Suppose that for some $k \geq 0$ the relation holds. Then, we have 
\begin{align}
\g_{k+1}=\g_{k}(1-c\g_{k}) \quad
& \Rightarrow \quad c\g_{k+1}=c\g_{k}(1-c\g_{k})\cr
 & \Rightarrow \quad c\g_{k+1}=\lambda_{k}(1-\lambda_{k})\cr
& \Rightarrow \quad \g_{k+1}=\lambda_{k+1}.
\end{align}
Hence, the result holds for $k+1$ implying that the result holds for all $k \ge 0$.
\end{proof}

Next, we show a relation for the sequences $\{\delta_k^*\}$ and $\{\Gamma_k^*\}$.

\begin{lemma}\label{lemma:min_maxm_relation}
Suppose that sequences $\{\delta_k^*\}$ and $\{\Gamma_k^*\}$ 
are given by relations \eqref{eqn:gmin0}--\eqref{eqn:gmink} and \eqref{eqn:gmax0}--\eqref{eqn:gmaxk} and $e_0 < \frac{2 \nu^2}{L^2}$. Then for all $k \geq 0$, {$\Gamma_k^*=(1+\beta)\delta_k^*$.}
\end{lemma}
\begin{proof}
Suppose that $\{\lambda_k\}$ is defined by $
\lambda_{k+1}=\lambda_{k}(1-\lambda_{k})$, for all$k\geq 0$,
where $\lambda_0=\frac{(\eta-\beta L)^2}{4(1+\beta)^2\nu^2}\,e_0$.
{In what follows,} we apply Lemma \ref{lemma:two-rec} twice to obtain
the result. {By the definition of $\lambda_0$ and $\delta_0^*$, we have
	that
	$\lambda_0=\frac{(\eta-\beta L)}{2}\delta_0^*$. Also, using $e_0
		<\frac{2\nu^2}{L^2}$ and definition of $\lambda_0$, we obtain
		\[\lambda_0 =\frac{(\eta- \beta L)^2}{4(1+\beta)^2\nu^2}e_0 <
		\frac{(\eta- \beta L)^2}{2(1+\beta)^2L^2}\leq
		\frac{\eta^2}{2L^2}<1.\] Therefore, the conditions of Lemma
		\ref{lemma:two-rec} hold for sequences
		$\{\lambda_k\}$ and $\{\delta_k^*\}$.
		Hence, Lemma \ref{lemma:two-rec} yields that for all $k \geq 0$,
\begin{align*}
& \lambda	_k =  \frac{(\eta-\beta L)}{2}\delta_k^*.
\end{align*}
Similarly, invoking Lemma \ref{lemma:two-rec} again, we have $\lambda	_k =\frac{(\eta-\beta L)}{2(1+\beta)}\Gamma_k^*$.
Therefore, from the two preceding relations, we can conclude the desired relation.}
Therefore, for all $k \geq 0$, $\Gamma_k^*=(1+\beta)\delta_k^*$.   
\end{proof}
The earlier set of results are essentially  adaptive rules for
determining the upper and lower bound of stepsize sequences, i.e.
$\{\delta_k^*\}$ and $\{\Gamma_k^*\}$. The
next proposition proposes recursive stepsize schemes for each player of game \eqref{eqn:problem}.  

\begin{proposition}{[Distributed adaptive steplength SA rules]}\label{prop:DASA}
Suppose that Assumption \ref{assum:different_main} and \ref{assum:w_k_bound} hold. Assume that set $X$ is bounded, i.e. there exists a positive constant $D \triangleq \max_{x,y \in X}\|x-y\|$. Suppose that the stepsizes for any player $i=1, \ldots,N$ are given by the following recursive equations
Suppose that Assumption \ref{assum:different_main} and \ref{assum:w_k_bound} hold. Assume that set $X$ is bounded, i.e. there exists a positive constant $D \triangleq \max_{x,y \in X}\|x-y\|$. Suppose that the stepsizes for any player $i=1, \ldots,N$ are given by the following recursive equations
\begin{align}
& \gamma_{0,i}=r_i\frac{c}{(1+\frac{\eta-2c}{L})^2\nu^2}\,D^2\label{eqn:gi0}
\\
& \gamma_{k,i}=\gamma_{k-1,i}\left(1-\frac{c}{r_i}\gamma_{k-1,i}\right) 
\quad \hbox{for all }k\ge 1.\label{eqn:gik}
\end{align}
where $r_i$ is an arbitrary parameter associated with $i$th player such
that $r_i \in [1, 1+\frac{\eta-2c}{L}]$, $c$ is an arbitrary fixed
constant $0<c < \frac{\eta}{2}$, $L$ is the Lipschitz constant of
mapping $F$, and $\nu$ is the upper bound given by Assumption
\ref{assum:w_k_bound} such that $D < \sqrt{2}\frac{\nu}{L}$. Then,
	{the following hold:}
\begin{itemize}
\item [(a)]  \an{$\frac{\g_{k,i}}{r_i}=\frac{\g_{k,j}}{r_j}$ for any $i,j=1,\ldots, N$ and $k \geq 0$.}
\item [(b)] Assumption \ref{assum:step_error_sub}b holds with $\beta=\frac{\eta -2c}{L}$, $\delta_k=\delta_k^*$, $\Gamma_k=\Gamma_k^*$, and $e_0=D^2$, where $\delta_k^*$ and $\Gamma_k^*$ are given by \eqref{eqn:gmin0}--\eqref{eqn:gmink} and \eqref{eqn:gmax0}--\eqref{eqn:gmaxk} respectively.
\item [(c)] The sequence $\{x_k\}$
generated by algorithm (\ref{eqn:algorithm_different}) converges \fy{a.s.}
to the unique solution of stochastic VI$(X,F)$.
\item [(d)] The results of Proposition \ref{prop:rec_results} hold for $\delta_k^*$ when $e_0=D^2$.  
\end{itemize}
\end{proposition}
\begin{proof}
\noindent (a) \
Consider the sequence $\{\lambda_k\}$ given by
\begin{align*}
& \lambda_0=\frac{c^2}{(1+\frac{\eta-2c}{L})^2\nu^2}\,D^2,
\\
& \lambda_{k+1}=\lambda_k(1-\lambda_k), 
\quad \hbox{for all }k\ge 1.
\end{align*}
Since for any ${i=1,\ldots, N}$, we have $\lambda_0=\frac{c}{r_i}\, \gamma_{0,i}$, using Lemma \ref{lemma:two-rec}, we obtain that for any $1 \leq i \leq N$ and $k \geq 0$,
\[\lambda_k=\frac{c}{r_i}\, \g_{k,i}.\] 
Therefore, for any $1 \leq i,j \leq N$, we obtain the desired relation in part (a).

\noindent (b) \ First we show that $\delta_k^*$ and $\Gamma_k^*$
are well defined. Consider the relation of part (a). Let $k\ge 0$ be
arbitrarily fixed. If $\g_{k,i}>\g_{k,j}$ for some $i \neq j$, then we
have $r_{i}>r_{j}.$ Therefore, the minimum possible $\g_{k,i}$ {is
	obtained} with $r_i=1$ and the maximum possible $\g_{k,i}$ {is
		obtained} with $r_i=1+\frac{\eta-2c}{L}$. Now, consider
		\eqref{eqn:gi0}--\eqref{eqn:gik}. If, $r_i=1$, and 
		$D^2$ is replaced by $e_0$, and $c$ by $\frac{\eta-\beta L}{2}$, we get the
		same recursive sequence defined by
		\eqref{eqn:gmin0}--\eqref{eqn:gmink}. Therefore, since the
		minimum possible $\g_{k,i}$ {is achieved} when $r_i=1$, we
		conclude that $\delta_k^* \leq \min_{i=1,\ldots,N} \g_{k,i}$ for
		any $k\ge 0$. This shows that $\delta_k^*$ is well-defined in
		the context of Assumption \ref{assum:step_error_sub}b.
		Similarly, it can be shown that $\Gamma_k^*$ is also
		well-defined in the context of  Assumption
		\ref{assum:step_error_sub}b. Now, Lemma~\ref{lemma:min_maxm_relation} implies that
		$\Gamma_k^*=(1+\frac{\eta-2c}{L})\delta_k^*$ for any $k \geq 0$,
		which shows that {Assumption} \ref{assum:step_error_sub}b is satisfied since $\beta=\frac{\eta-2c}{L}$ and $0<c<\frac{\eta}{2}$.

\noindent (c) In view of Proposition \ref{prop:rel_bound}, to show the
almost-sure convergence, it suffices to show that Assumption
\ref{assum:step_error_sub} holds. Part (b) implies that Assumption
\ref{assum:step_error_sub}b holds for the specified choices. Since
$\g_{k,i}$ is a recursive sequence for each $i$, Assumption
\ref{assum:step_error_sub}a holds using Proposition 3 in \cite{Farzad1}.

\noindent (d) Since $D < \sqrt{2}\frac{\nu}{L}$, it follows that $e_0 <\frac{2\n^2}{L^2}$, which shows that the conditions of Proposition \ref{prop:rec_results} are satisfied.
\end{proof}
\fy{
\section{Numerical results}\label{sec:numerical}
In this section, we report the results of our numerical experiments on a stochastic bandwidth-sharing problem in
communication networks (Sec.  \ref{sec:traffic}). We
	compare the performance of the distributed adaptive stepsize SA
	scheme (DASA) given by (\ref{eqn:gi0})--(\ref{eqn:gik}) with that of
	SA schemes with harmonic stepsize sequences (HSA), where agents use
	the stepsize $\frac{\theta}{k}$ at iteration $k$. More precisely, we
	consider three different values of the parameter $\theta$, i.e.,
	$\theta = 0.1$, $1$, and $10$. This diversity of choices allows us to observe
	the sensitivity of the HSA scheme to different settings of the
	parameters.
	   
\subsection{A bandwidth-sharing problem in computer networks}\label{sec:traffic}
We consider a communication network where users compete for the
bandwidth. Such a problem can be captured by an optimization framework (cf. \cite{Cho05}). Motivated by this model, we consider a network
with $16$ nodes, $20$ links and $5$ users. Figure \ref{fig:network}
shows the configuration of this network. \begin{figure}[htb]
\begin{center}
 \includegraphics[scale=.30]{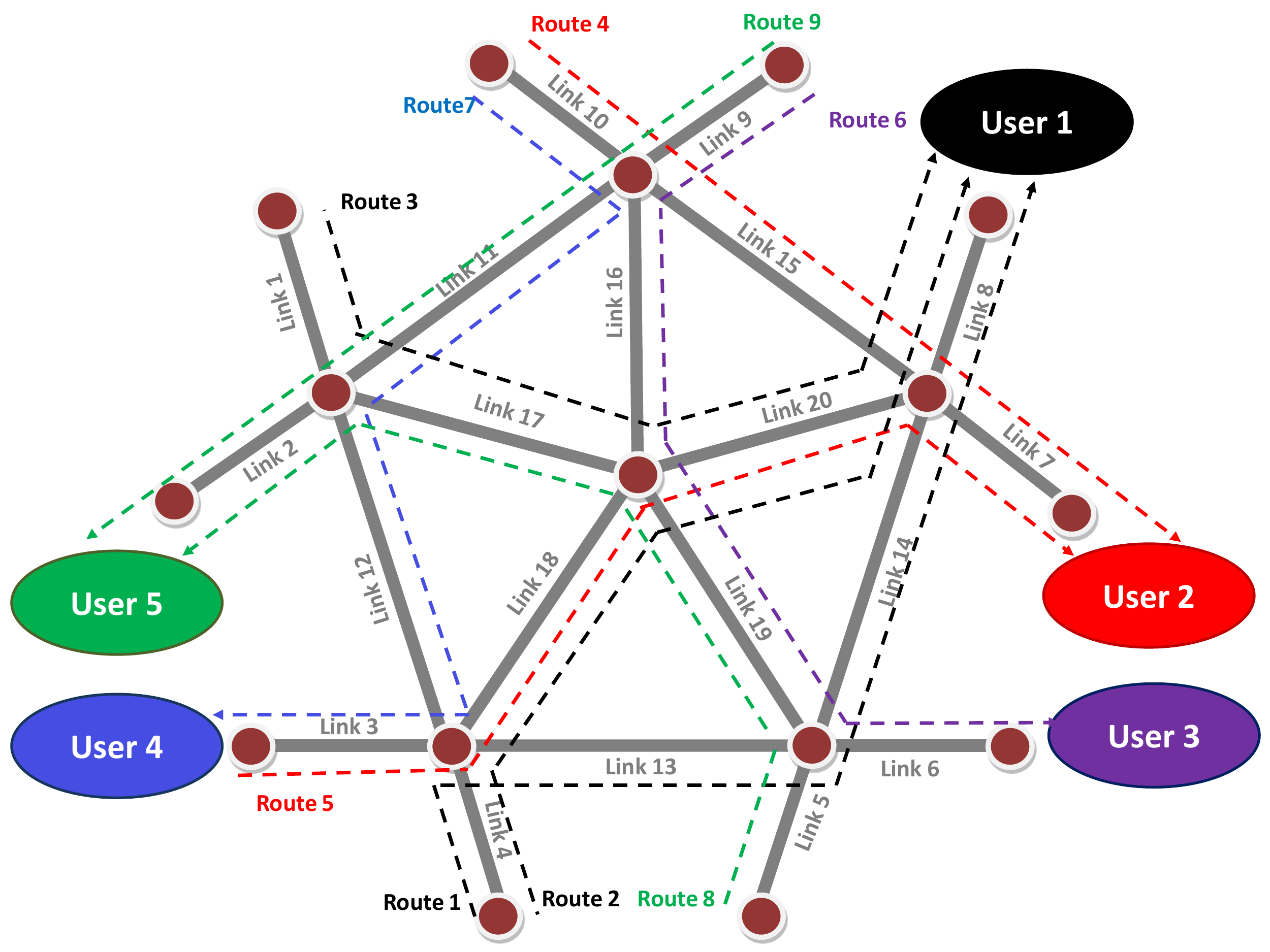}
 \caption{The network}
  \label{fig:network}
   \end{center}
 \vspace{-0.1in}  
 \end{figure} Users have access to  different routes as shown in Figure
 \ref{fig:network}. For example, user $1$ can  access routes $1$, $2$, and $3$. 
 Each user is characterized by a cost function. Additionally, there is a
congestion cost function that depends on the aggregate flow. More
specifically, the cost function
user $i$ with flow rate (bandwidth) $x_i$ is defined by
\[f_i(x_i,\xi_i)\triangleq -\sum_{r \in \mathcal{R}(i)} \xi_{i}(r)\log(1+x_i(r)),\]
for $i=1,\ldots, 5$, where $x\triangleq (x_1; \ldots; x_5)$ is the flow decision vector of the users, $\xi \triangleq (\xi_1; \ldots; \xi_5)$ is a random parameter corresponding to the different users, 
$\mathcal{R}(i)=\{1,2,\ldots,n_i\}$ is the set of routes assigned to the $i$-th user, $x_{i}(r)$ and $\xi_{i}(r)$ 
are the $r$-th element of the decision vector $x_i$ and the random vector $\xi_i$, respectively. 
We assume that
$\xi_i(r)$ is drawn from a uniform distribution for each $i$ and $r$ and the links have limited capacities given by $b$.

We may define the routing matrix $A$ that describes the relation between set of routes $\mathcal{R}=\{1,2,\ldots,9\}$ and set of links $\mathcal{L}=\{1,2,\ldots,20\}$. Assume that $A_{lr}=1$ if route $r \in \mathcal{R}$ goes through link $l \in \mathcal{L}$ and $A_{lr}=0$ otherwise. Using this matrix, the capacity constraints of the links can be described by $Ax \leq b$.

We formulate this model as a stochastic optimization problem given by
\begin{align}\label{eqn:network-prob}
\displaystyle \mbox{minimize} \qquad &  \sum_{i=1}^N \EXP{f_i(x_i,\xi_i)} +c(x) \\
\mbox{subject to} \qquad& Ax \leq b, \hbox{ and } x \geq 0,\notag
\end{align}
where $c(x)$ is the network congestion cost.
We consider this cost of the form $c(x)=\|Ax\|^2$. 
Problem (\ref{eqn:network-prob}) is a convex optimization problem and the optimality conditions can be stated as a variational inequality given by $\nabla f(x^*)^T(x-x^*) \geq 0$, where $f(x) \triangleq \sum_{i=1}^N \EXP{f_i(x_i,\xi_i)} +c(x)$. Using our notation in Sec. \ref{sec:formulation}, we have $$F(x)=-\left( 
\frac{\bar \xi_1(1)}{1+x_1(1)};\ldots;
\frac{\bar\xi_5(2)}{1+x_5(2)} \right)+2A^TAx,$$
where $\bar \xi_i(r_i) \triangleq \EXP{\xi_i(r_i)}$ for any $i=1,\ldots,5$ and $r_i=1,\ldots,n_i$. It can be shown
that the mapping $F$ is strongly monotone and Lipschitz with specified parameters (cf. \cite{Farzad03}). We solve the bandwidth-sharing problem for $12$
different settings of parameters shown in Table
\ref{tab:network_errors1}. We consider $4$ parameters in our model that
scale the problem. Here, $m_b$ denotes the multiplier of the capacity
vector $b$, $m_c$ denotes the multiplier of the congestion cost function
$c(x)$, and $m_\xi$ and $d_\xi$ are two multipliers that parametrize the
random variable $\xi$. $S(i)$
denotes the $i$-th setting of parameters. For each of these $4$
parameters, we consider $3$ settings where one parameter changes and
other parameters are fixed. This allows us to observe the sensitivity of
the algorithms with respect to each of these parameters. 
\begin{table}[htb] 
\vspace{-0.05in} 
\tiny 
\centering 
\begin{tabular}{|c|c|c|c|c|c|} 
\hline 
-&S$(i)$ &  $m_b$ & $m_c$ & $m_\xi$ & $d_\xi$
\\ 
\hline 
\hline
{$m_b$} &1 &  1 & 1 & 5 & 2
  \\

\hbox{ }& 2 &  0.1 & 1 & 5 & 2
  \\

\hbox{ }& 3 &  0.01 & 1 & 5 & 2  \\ 
\hline 

$m_c$ &4 &  0.1 & 2 & 2 & 1
  \\

\hbox{ }& 5 &  0.1 & 1 & 2 & 1
  \\

\hbox{ }& 6 &  0.1 & 0.5 & 2 & 1 \\ 
\hline 

$m_\xi$ &7 &  1 & 1 & 1 & 5
  \\

\hbox{ }& 8 &  1 & 1 & 2 & 5
  \\

\hbox{ }& 9 &  1 & 1 & 5 & 5 \\ 
\hline 
$d_\xi$ &10 &   1 & 0.01 & 1 & 1
  \\

\hbox{ }& 11 & 1 & 0.01 & 1 & 2
  \\

\hbox{ }& 12 &  1 & 0.01 & 1 & 5\\
\hline

\end{tabular} 
\caption{Parameter settings} 
\label{tab:network_errors1} 
\vspace{-0.1in} 
\end{table}	
The SA algorithms are terminated after $4000$ iterates.
To measure the error of the schemes, we run each scheme $25$ times and
then compute the mean squared error (MSE) using the metric $\frac{1}{25}\sum_{i=1}^{25}\|x_k^i-x^*\|^2$ for any $k=1,\ldots,4000$, where $i$ denotes the $i$-th sample.
Table \ref{tab:Traffic2} and \ref{tab:Traffic3} show the $90\%$ confidence intervals (CIs) of the error for the DASA and HSA schemes. 
\begin{table}[htb] 
\vspace{-0.05in} 
\tiny 
\centering 
\begin{tabular}{|c|c||c||c|} 
\hline 
-&S$(i)$ &   DASA - $90\%$ CI &  HSA with $\theta=0.1$- $90\%$ CI 
\\ 
\hline 
\hline
  {$m_b$} &1 & [$2.97 $e${-6}$,$4.66 $e${-6}$] &  [$1.52 $e${-6}$,$2.37 $e${-6}$] 
  \\

\hbox{ }&2 & [$2.97 $e${-6}$,$4.66 $e${-6}$] &  [$1.52 $e${-6}$,$2.37 $e${-6}$] 
  \\

\hbox{ }&3& [$1.15 $e${-7}$,$3.04 $e${-7}$] &  [$2.12 $e${-8}$,$4.92 $e${-8}$] 
  \\ 
  
  \hline 
  
  $m_c$ &4 &  [$4.39 $e${-7}$,$6.55 $e${-7}$] &  [$1.33$e${-6}$,$1.80 $e${-6}$] 
  \\

\hbox{ }&5 &  [$1.29 $e${-6}$,$1.97 $e${-6}$] &  [$9.00$e${-6}$,$1.20 $e${-5}$] 
  \\

\hbox{ }&6 &  [$3.44 $e${-6}$,$5.36 $e${-6}$] &  [$2.26$e${-4}$,$2.53 $e${-4}$] 
  \\ 
  \hline

 $m_\xi$ &7& [$4.29 $e${-5}$,$6.40 $e${-5}$] &  [$7.92 $e${-5}$,$1.49 $e${-4}$] 
  \\

\hbox{ }&8 &[$3.18 $e${-5}$,$4.83 $e${-5}$] &  [$3.46 $e${-5}$,$6.07 $e${-5}$] 
  \\

\hbox{ }&9 & [$1.83 $e${-5}$,$2.88 $e${-5}$] &  [$6.12 $e${-6}$,$9.99 $e${-6}$] 
  \\
  \hline
  
  $d_\xi$ &10 &[$3.82 $e${-4}$,$5.91 $e${-4}$] &  [$2.86 $e${+1}$,$2.86 $e${+1}$] 
  \\ 
 
\hbox{ }&11 & [$9.81 $e${-4}$,$1.44 $e${-3}$] &  [$2.86 $e${+1}$,$2.86 $e${+1}$] 
  \\  
\hbox{ }&12 &[$6.26 $e${-3}$,$8.44 $e${-3}$] &  [$2.85 $e${+1}$,$2.86 $e${+1}$] 
  \\ 

  \hline

\end{tabular} 
\caption{$90\%$ CIs for DASA and HSA schemes -- Part I} 
\label{tab:Traffic2} 
\vspace{-0.1in} 
\end{table}

\begin{table}[htb] 
\vspace{-0.05in} 
\tiny 
\centering 
\begin{tabular}{|c|c||c||c|} 
\hline 
-&S$(i)$ &  HSA with $\theta=1$ - $90\%$ CI & HSA with $\theta=10$ - $90\%$ CI
\\ 
\hline 
\hline
  {$m_b$} &1  &  [$1.70 $e${-6}$,$2.97 $e${-6}$]&  [$1.33 $e${-5}$,$1.81 $e${-5}$]
  \\

\hbox{ }&2  &  [$1.70 $e${-6}$,$2.97 $e${-6}$]&  [$1.33 $e${-5}$,$1.81 $e${-5}$]
  \\

\hbox{ }&3 &  [$4.66 $e${-8}$,$1.17 $e${-7}$]&  [$8.07 $e${-7}$,$2.43 $e${-6}$]
  \\ 
  
  \hline 
  
  $m_c$ &4  &  [$4.71 $e${-7}$,$8.75 $e${-7}$]&  [$3.84 $e${-6}$,$5.38 $e${-6}$]
  \\

\hbox{ }&5  &  [$7.88 $e${-7}$,$1.36 $e${-6}$]&  [$5.61 $e${-6}$,$7.98 $e${-6}$]
  \\

\hbox{ }&6 &  [$1.25 $e${-6}$,$1.99 $e${-6}$]&  [$7.34 $e${-6}$,$1.12 $e${-5}$]
  \\ 
  \hline

 $m_\xi$ &7 &  [$2.83 $e${-5}$,$4.75 $e${-5}$]&  [$1.84 $e${-4}$,$2.75 $e${-4}$] 
  \\

\hbox{ }&8  &  [$1.97 $e${-5}$,$3.39 $e${-5}$]&  [$1.40 $e${-4}$,$1.99 $e${-4}$] 
  \\

\hbox{ }&9  &  [$1.06 $e${-5}$,$1.85 $e${-5}$]&  [$8.33 $e${-5}$,$1.13 $e${-4}$]
  \\
  \hline
  
  $d_\xi$ &10 &  [$5.50 $e${-1}$,$5.70 $e${-1}$]&  [$7.23 $e${-5}$,$9.64 $e${-5}$]
  \\ 
 
\hbox{ }&11&  [$5.45 $e${-1}$,$5.85 $e${-1}$]&  [$2.85 $e${-4}$,$3.80 $e${-4}$]
  \\  
\hbox{ }&12 &  [$5.47 $e${-1}$,$6.44 $e${-1}$]&  [$1.77 $e${-3}$,$2.36 $e${-3}$]
  \\ 

  \hline

\end{tabular} 
\caption{$90\%$ CIs for DASA and HSA schemes -- Part II} 
\label{tab:Traffic3} 
\vspace{-0.1in} 
\end{table}	

 \textbf{Insights:} 
We observe that DASA scheme performs favorably and is far more robust in
comparison with the HSA schemes with different choice of $\theta$.
Importantly, in most of the settings, DASA stands close to the HSA
scheme with the minimum MSE.  Note that when $\theta=1$ or $\theta=10$,
the stepsize $\frac{\theta}{k}$ is not within the interval
$(0,\frac{\eta-\beta L}{(1+\beta)^2L^2}]$ for small $k$ and is not
		feasible in the sense of Prop.  \ref{prop:rec_results}.
		Comparing the performance of each HSA scheme in different
		settings, we observe that HSA schemes are fairly sensitive to
		the choice of parameters. For example, HSA with $\theta=0.1$
		performs very well in settings S$(1)$, S$(2)$, and S$(3)$, while
		its performance deteriorates in settings S$(10)$, S$(11)$, and
		S$(12)$. A similar discussion holds for other two HSA schemes. A
		good instance of this argument is shown in Figure \ref{fig:traffic_all_s4} and \ref{fig:traffic_all_s11}.

  \begin{figure}[htb]
 \centering
\label{fig:traffic_prob4}\includegraphics[scale=.33]{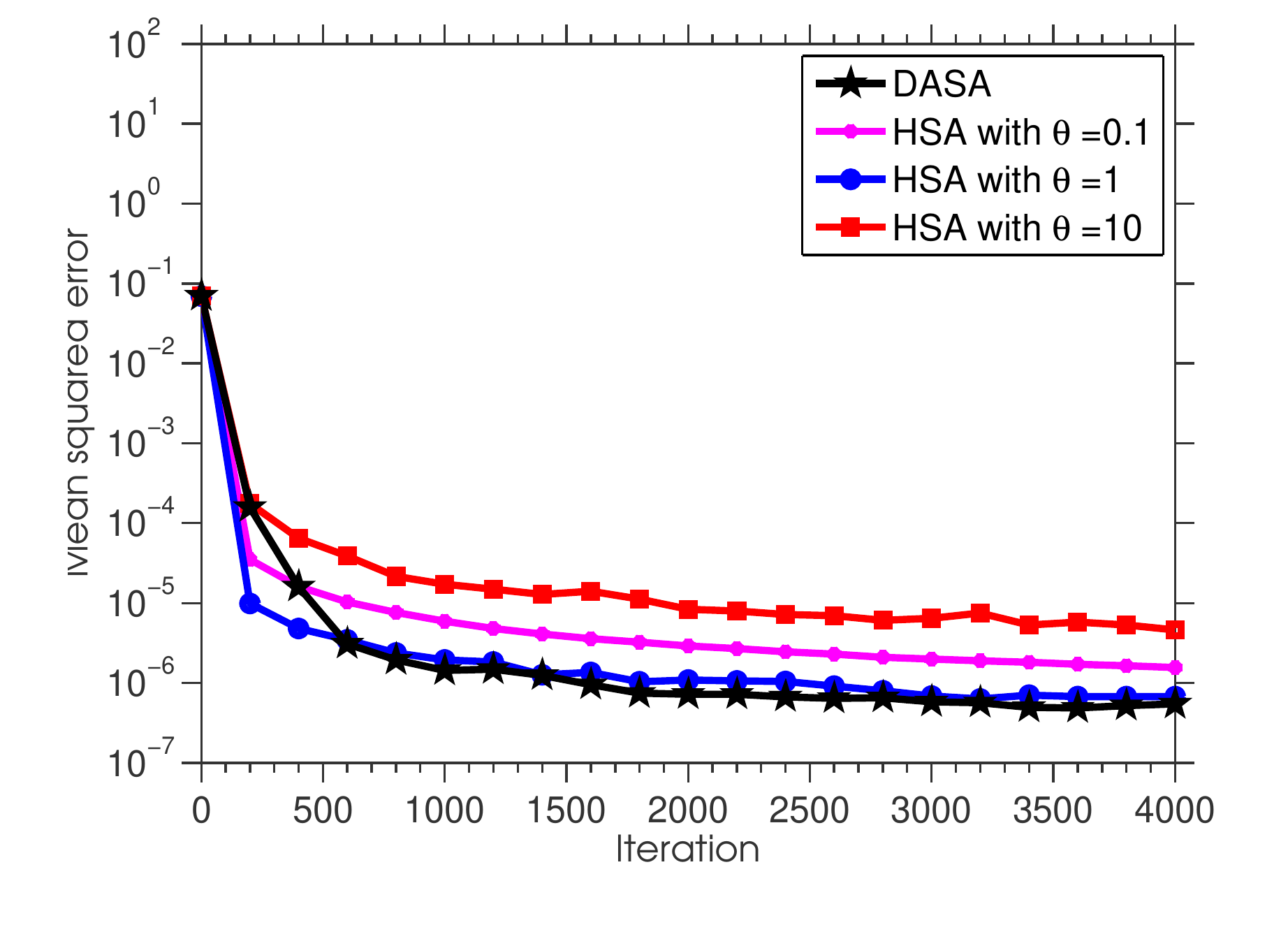}
\caption{DASA vs. HSA schemes -- Setting S($4$)}
\label{fig:traffic_all_s4}
\end{figure}
\vspace{-.1in}
  \begin{figure}[htb]
 \centering
\label{fig:traffic_prob11}\includegraphics[scale=.33]{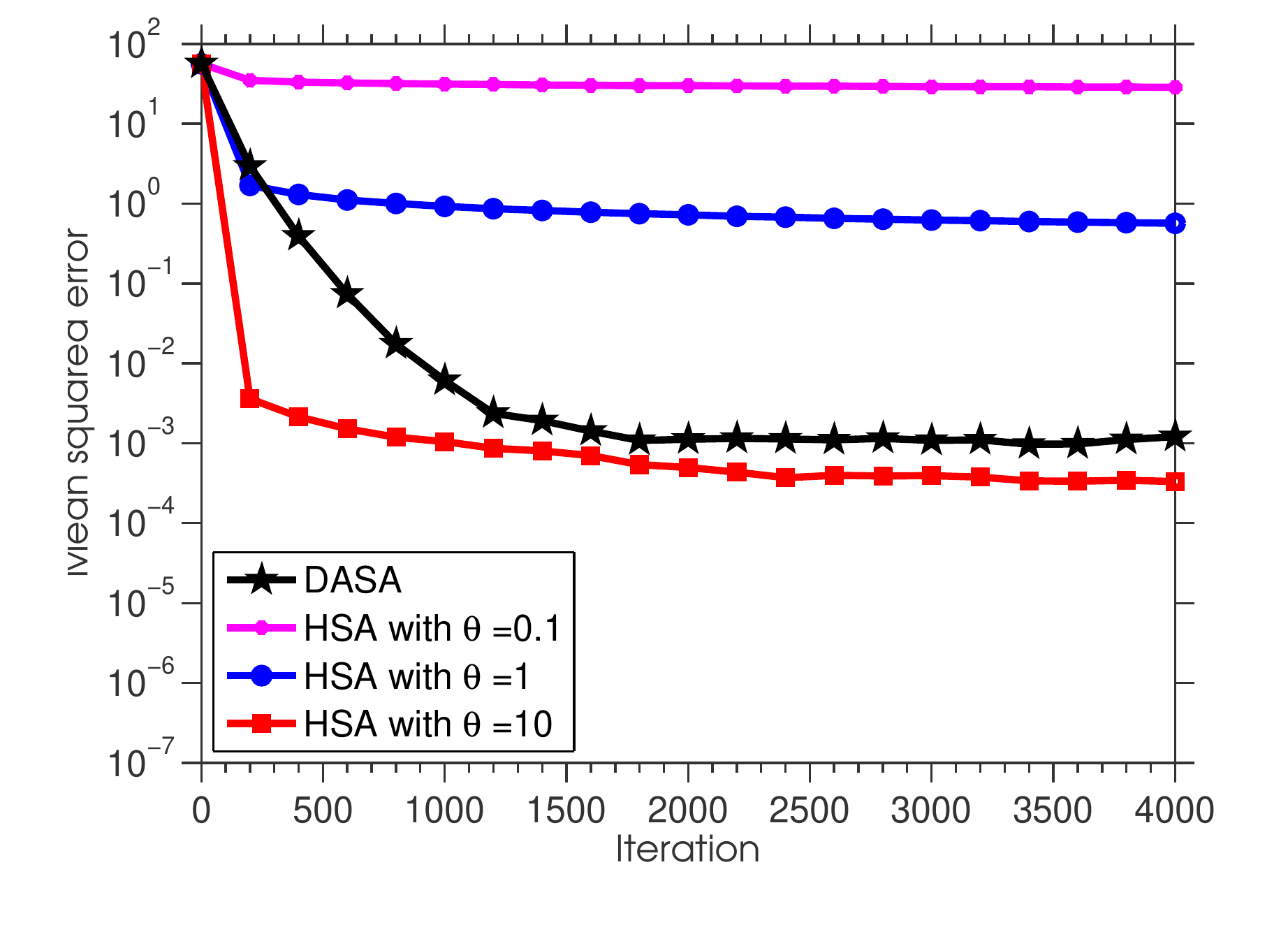}
\caption{DASA vs. HSA schemes -- Setting S($11$)}
\label{fig:traffic_all_s11}
\end{figure}}
\vspace{-.1in}
\section{Concluding remarks} 
We considered distributed monotone stochastic Nash games where each player minimizes a convex function 
on a closed convex set. We first formulated the problem as a stochastic
VI and  then showed that under suitable conditions, for a strongly monotone and Lipschitz mapping, the SA scheme guarantees almost-sure convergence to the solution. 
Next, motivated by the naive stepsize choices of SA algorithm, we proposed a class of distributed adaptive steplength rules where each player can choose his own stepsize independent of 
the other players from a specified range. We showed that this scheme provides 
almost-sure convergence and also minimizes a suitably defined error bound of the SA algorithm. 
Numerical experiments, reported in Section \ref{sec:numerical} confirm this conclusion.

\bibliographystyle{IEEEtran}
\bibliography{Nash_adapt}
\end{document}